\documentclass[a4paper,11pt,reqno]{amsart}
\usepackage[margin=1in]{geometry}
\usepackage{amsmath,amsthm,amssymb,amsfonts,extarrows,mathrsfs,comment,tikz-cd}
\usepackage{stmaryrd}
\usepackage{dsfont}
\usepackage{mathtools}
\usepackage{enumerate}
\usepackage[all]{xy}
\usepackage{indentfirst}
\usepackage{mathtools}
\usepackage{setspace}

\usepackage{graphicx,epstopdf}

\newtheorem{theorem}[subsection]{Theorem}
\newtheorem{proposition}[subsection]{Proposition}

\newtheorem{corollary}[subsection]{Corollary}

\theoremstyle{definition}
\newtheorem{remark}[subsection]{Remark}

\newtheorem{definition}[subsection]{Definition}

\newcommand\cC{\mathcal{C}}

\newcommand\cE{\mathcal{E}}
\newcommand\cF{\mathcal{F}}

\newcommand\cO{\mathcal{O}}

\newcommand\cW{\mathcal{W}}

\newcommand\CC{\mathbb{C}}

\newcommand\ZZ{\mathbb{Z}}

\newcommand\bG{\mathbf{G}}

\newcommand\rD{\mathrm{D}}

\newcommand\rG{\mathrm{G}}
\newcommand\rH{\mathrm{H}}

\newcommand\rR{\mathrm{R}}
\newcommand\rS{\mathrm{S}}

\DeclareMathOperator{\CH}{CH}

\DeclareMathOperator{\id}{id}

\DeclareMathOperator{\NS}{NS}
\DeclareMathOperator{\Hom}{Hom}
\DeclareMathOperator{\Ext}{Ext}

\DeclareMathOperator{\Gr}{Gr}

\begin{document}

\title{Beauville-Voisin filtrations on zero cycles of moduli space of stable sheaves on K3 surfaces}
\date{\today}

\author{Zhiyuan Li}
\address{Zhiyuan Li, Shanghai Center for Mathematical Sciences, Fudan University, Jiangwan Campus, Shanghai, 200438, China}
\email{zhiyuan\_li@fudan.edu.cn}
\author{Ruxuan Zhang}
\address{Ruxuan Zhang, Shanghai Center for Mathematical Sciences, Fudan University, Jiangwan Campus, Shanghai, 200438, China}
\email{rxzhang18@fudan.edu.cn}

\begin{abstract}
The Beauville-Voisin conjecture predicts the existence of a filtration on projective hyper-K\"ahler manifolds opposite to the conjecture Bloch-Beilinson filtration, called the Beauivlle-Voisin filtration. In \cite{Voi16}, Voisin has introduced a filtration on zero cycles of an arbitrary projective hyper-K\"ahler manifold. On moduli space of stable objects of a projective K3 surface, there are other candidates constructed by Shen-Yin-Zhao, Barros-Flapan-Marian-Silversmith in \cite{SYZ20, BFMS19} and  more recently by  Vial in \cite{Vi20} from different point of views. According to the work in \cite{Vi20},  all of them are proved to be equivalent except Voisin's filtration. In this paper, we show that Voisin's filtration is the same as the other  filtrations.   As an application, we prove a  conjecture in \cite{BFMS19}. 
\end{abstract}

\subjclass[2020]{14J28; 14F08; 14J42}
\keywords{Beauville-Voisin filtration, moduli space of sheaves, hyper-K\"ahler}

\maketitle

\section{Introduction}
On a projective K3 surface $X$,  Beauville and Voisin \cite{BV04} showed that $X$ carries a canonical zero-cycle class of degree 1,
$$[o_X] \in  \CH_0(X),$$
where $o_X$ can be taken any point lying on a rational curve in $X$.  The intersections of divisor classes in $X$, as well as the second Chern class of $X$, lie in $\ZZ\cdot  [o_X]$.   In \cite{OG13}, O’Grady introduced a  filtration $\rS_\bullet(X)$ on $\CH_0(X)$,
$$ \rS_i(X) =\bigcup_{\deg([z])=i}\{ [z] + \ZZ \cdot [o_X] \}$$ for all effective zero-cycles $z$ of degree $i$, which gives another characterization of the Beauville-Voisin class $[o_X]$. 
More generally, Beauville-Voisin's conjecture predicts that a similar picture also holds for any projective hyper-K\"ahler manifold $M$ and there exists a  filtration (called Beauville-Voisin filtration)  on the Chow ring viewed as an opposite to the conjectural Bloch-Beilinson filtration. For zero cycles,  Voisin has constructed a filtration $\rS_\bullet\CH_0(M)\subseteq \CH_0(M)$   in \cite{Voi16}  defined by  $$\rS_{i}\CH_0(M):=\left<x\in M|~\dim O_x\geq d-i\right>$$ where $O_x$ is the orbit of $x$ under the rational equivalence.   

In the case $M=M_\sigma(v)$ is the moduli space of $\sigma$-stable objects in the derived category $\rD^b(X)$ of a projective K3 surface $X$ for some Bridgeland stability condition $\sigma$.  Shen-Yin-Zhao  discovered a  filtration on the derived category $\rD^b(X)$, which gives another natural filtration $\rS_{\bullet}^{\rm SYZ}\CH_0(M)$, $$\rS_i^{\rm SYZ}\CH_0(M):= \left<E|~c_2(E)\in \rS_i(X)\right>. $$ 
They  showed that  there is an inclusion 
$$\rS_{i}^{\rm SYZ} \CH_0(M)\subseteq \rS_{i}\CH_0(M),$$ 
between the two filtrations.  A natural question is  whether the two filtration coincides (cf.~\cite[Remark 2.6]{SYZ20}).    
This has been confirmed on Hilbert scheme of points on $X$ via using the density of constant cycle subvarieties. In this paper, we give an affirmative answer to this question and the main result is

\begin{theorem}\label{mainthm}
Let $M=M_\sigma(v)$ be the moduli space of $\sigma$-stable objects in $\rD^b(X)$ with a primitive Mukai vector $v$. Suppose $\sigma$ is $v$-generic. Then  we have    
$$\rS_{\bullet}^{\rm SYZ} \CH_0(M)= \rS_{\bullet}\CH_0(M).$$
\end{theorem}

As a corollary, we obtain the density of constant cycle subvarieties on $M$ (See Corollary \ref{cor:dense}). Finally, let us compare Voisin's filtration with other filtrations.  In \cite{BFMS19}, Flapan, Marian and Silversmith have constructed two filtrations $\rS_\bullet^{\rm small}\CH_0(M)$ and $\rS^{\rm BFMS}_\bullet \CH_0(M)$ from the product point of view, while Vial has found a so called {\it co-radical} filtration  $\rR_\bullet \CH_0(M)$ from the co-algebra structure on birational motives. Combined with the results in \cite{BFMS19, Vi20}, our main theorem indicates that 
\begin{theorem}\label{mainthm2} With the assumptions in Theorem \ref{mainthm}, we have 
  \begin{equation}\label{BFMS=SYZ}
    \rS^{\mathrm{small}}_{\bullet}\CH_0(M)=\rS^{\rm BFMS}_\bullet \CH_0(M)=\rS_{\bullet}\CH_0(M)=\rS_\bullet^{\rm SYZ} \CH_0(M)=\rR_\bullet \CH_0(M).
\end{equation}
As a consequence, \cite[Conjecture 4]{BFMS19} holds. 
\end{theorem}

\subsection*{Acknowledgement} The authors are grateful to Junliang Shen, Qizheng Yin and Ziyu Zhang for useful comments. The authors were supported by NSFC for Innovative Research Groups (Grant No.~12121001) and partially supported by NSFC grant for General Program (11771086).

\section{Degenerate loci of vector bundles}
In this section, we recap  Voisin's work in \cite{Voi15} which establish a  relation between the zero cycles on the moduli space of stable sheaves and the  Hilbert scheme of points. 
Let $(X,H)$ be a polarized K3 surface and let $M=M_H(v)$ be the moduli space of Gieseker $H$-stable sheaves on $X$ with a primitive  Mukai vector $v=(r,D, s)$. We may assume $H$ is $v$-generic so that $M$ is a smooth projective hyper-K\"ahler variety of dimension 
$$v^2+2=D^2-2rs+2.$$ 
Given a locally free sheaf $E \in M$, let  $d=\frac{1}{2}\Ext^1(E,E)$ and $g=\dim \Gr(r-1,\rH^0(X,E))$. As in \cite{Voi15}, if $E$ is globally generated and $\rH^1(X,E)=0$, we have an exact sequence for a general $(r-1)$-dimensional sections of $\rH^0(X,E)$ 
\begin{equation}\label{degeneracy}
0 \rightarrow \mathcal{O}_X^{r-1} \xrightarrow{s} E \rightarrow I_Z \otimes D \rightarrow 0
\end{equation}
where $Z$ is the degeneracy locus of $s$ consisting of $d+g$ distinct points. 
Then the sequence defines a rational map
$$\phi_E:\rG_E\dashrightarrow X^{[d+g]},$$
where $\rG_E=\Gr(r-1,\rH^0(X,E))$, sending $s$ to $I_Z\otimes D\subset X^{[d+g]}$. 

\begin{proposition}\label{unique}
Assume the  condition $(\star)$
\begin{equation*}\label{*}
   E~\hbox{is globally generated and}~ \rH^i(X,E)=0~ \hbox{for} ~i=1,2.
\end{equation*}
holds. 
The morphism $\phi_E$ is generically injective.
Moreover, for any two stable sheaves $E, F\in M$ satisfying $(\star)$ as extensions of $\cO_X^{r-1}$ and $I_Z\otimes D$ for some $Z$, we have $E\cong F$.
\end{proposition}
\begin{proof}
The generic injectivity of $\phi_E$ is already proved in \cite[Proposition 3.2]{Voi15}. It follows from the computation of $\Hom(E,I_Z\otimes D)$. Indeed, apply  the functor $\Hom(E,-)$ to the exact sequence  \eqref{degeneracy}, one can obtain a long exact sequence:
$$0\rightarrow \Hom(E,\cO_X^{r-1})\rightarrow \Hom(E,E)\rightarrow \Hom(E,I_Z \otimes D)\rightarrow \Ext^1(E,\cO_X^{r-1})\to \ldots$$
By the Serre duality, we have $$\Hom(E,\cO_X^{r-1})=\Ext^1(E,\cO_X^{r-1})=0. $$ Since $E$ is simple, $\Hom(E,E)\cong \CC$.  It follows that $$\Hom(E,I_Z \otimes D)\cong \CC.$$ It implies that $\phi_E$ is injective on $s$ whenever we have the sequence \eqref{degeneracy} for $s$. If we have  another extension
$$ 0 \rightarrow \mathcal{O}_X^{r-1} \rightarrow F \rightarrow I_Z \otimes D \rightarrow 0,$$
after applying $\Hom(E,-)$, one can find that
$$\Hom(E,F)\cong \Hom(E,I_Z \otimes D)\cong \CC.$$
In particular, if $F$ is a stable sheaf, then $E\cong F$.
\end{proof}

More generally, one can extend this construction to a family of stable locally free sheaves. Let $\cE$ be a locally free sheaf on $B\times X$ for some variety $B$ satisfying that for $E_b=\cE|_{\{b\}\times X}$,  we have $$\rH^i(X,E_b)=0,~i=1,2$$ for any $b\in B$. In particular, the spaces of sections $\rH^0(X,E_b)$ are of the same dimension and $\cF:=\pi_*\cE$ is a locally free sheaf on $B$, where $\pi:X\times B\rightarrow B$ is the projection.
Let $\bG_B:=\Gr^{r-1}(\cF)$ be the Grassmannian of rank $r-1$  subbundles of $\cF$ which fits into the following diagram:
\begin{equation}\label{diag1}
    \begin{tikzcd}
\bG_B\times X \arrow[r, "p\times \id"] \arrow[d, "q"] & B\times X \arrow[d, "\pi"] \\
\bG_B \arrow[r, "p"]                                 & B                         
\end{tikzcd}
\end{equation}

There exists a universal locally free subsheaf $\cW\hookrightarrow p^*\cF$ on $\bG_B$. Pulling back  to $\bG_B\times X$, one has a  natural morphism:
\begin{equation}\label{relative}
   \sigma: q^*\cW\rightarrow (p\times \id)^*\cE. 
\end{equation}
Then the degeneracy locus of $\sigma$ in $\bG_B\times X$ is:
$$D_k(\sigma)=\{(s,x)~|~\mathrm{rank}(\sigma(s,x))\leq k\}.$$
A general fiber of $D_k(\sigma)\rightarrow \bG_B$ consists of $d+g$ distinct points in $X$ by Proposition \ref{unique}.
Then it defines a rational map $\phi$:
\begin{equation*}
\begin{tikzcd}
\phi_B: \bG_B \arrow[r,dashed] &  X^{[d+g]} 
\end{tikzcd}
\end{equation*}
The restriction of $\sigma$ at $\{(b,s)\}\times X\subset \bG_B\times X$ is exactly  $\cO_X^{r-1}\xrightarrow{s} E_b$ in Voisin's construction. In this way, one can obtain a rational map
\begin{equation}
   \phi_B: \bG_B\dashrightarrow X^{[d+g]}
\end{equation}
whose restriction to the fiber  $q^{-1}(E)\subseteq \bG_B$ is  $\phi_E$.

\section{Proof of Main theorem}

 Let $v=(r,D,s)$  be a primitive Mukai vector with $r\geq 0$ and assume that $H$ is $v$-generic.  We first consider the locally free sheaves in  $M_H(v)$.
\begin{theorem}\label{thm:lf}
  Let $E$ be a  locally free sheaf in $M=M_H(v)$.  Then $c_2(E)\in \rS_i(X)$ if there exists a constant cycle subvariety $Y\subseteq M$ such that $[E]\in Y$ and $\dim Y\geq d-i$. In particular, if $M$ consists of stable locally free sheaves, then $\rS_\bullet^{\rm SYZ} \CH_0(M)=\rS_\bullet\CH_0(M)$.
\end{theorem}

\begin{proof}
After tensoring line bundles $mH$ for sufficiently large $m$ without changing $M$, we may assume  that the sheaves in $M$ satisfy the condition $(\star)$ by \cite[Theorem 0.1, Section 10.1]{langer2006moduli}.   There exists an $1\boxtimes \alpha$-twisted universal sheaf on $X\times M$. We take a Severi-Brauer variety $$\nu:P\rightarrow M$$ for the Brauer class $\alpha$. Then there is a universal sheaf $\cE$ on $X\times P$ in the sense that for $x\in P$, the stalk $\cE_x$ of $\cE$ is isomorphic to $E_{\nu (x)}$, where $E_{\nu (x)}$ is the sheaf corresponding to $\nu (x)\in M$.

With the same notations in the previous section,   for any subvariety $Y\subseteq M$ containing $E$, take  $B=\nu^{-1}(Y)$,  we  can apply relative version of Voisin's construction to obtain a rational map 
$$\phi_{\nu^{-1}(Y)}:\bG_{\nu^{-1}(Y)} \dashrightarrow X^{[d+g]}$$
By Proposition \ref{unique}, we have 
\begin{enumerate}
    \item the restriction of $\phi_{\nu^{-1}(Y)}$ to 
the fiber of $p: \bG_{\nu^{-1}(Y)}\rightarrow \nu^{-1}(Y)$ is generic injective and hence $\phi_{\nu^{-1}(Y)}$ sends each fiber  to a $g$-dimensional constant cycle subvariety in $X^{[d+g]}$.
\item for different $E'\in Y$ and the images $\phi(p^{-1}(E'))$ in $X^{[d+g]}$ are disjoint if $E'$ is locally free.
\end{enumerate}   Moreover, for any two points  $s_1,s_2\in \nu^{-1}(E')$, their corresponding sheaves are the same and so are their images in $X^{[d+g]}$. Then we have 
$$\mathrm{dim}~ \phi(p^{-1}(Y))=\dim Y+g.$$

Now, suppose $Y$ is  a constant cycle subvariety on $M$ of dimension $\geq d-i$.  According to the main Theorem in \cite{MZ20}, two sheaves  $E', E''\in M$ are rationally equivalent in $\CH_0(M)$ if and only if $c_2(E')$ and $c_2(E'')$ are rationally equivalent in $\CH_0(X)$.  For any point in $\bG_{\nu^{-1}(Y)}$,  its image $I_Z\otimes D$ represent the same class in $\CH_0(X^{[d+g]})$ as $$c_2(I_Z\otimes D)=c_2(E) $$ in $\CH_0(X)$ because the points in $Y$ are constant in $\CH_0(M)$. This implies that the Zariski closure of the image $\phi(p^{-1}(Y))$ is a constant cycle subvariety on $X^{[d+g]}$. By \cite[Theorem 2.1]{Voi15},  we have 
$$c_2(I_Z\otimes D)\in \rS_i(X)$$
 and hence $E\in \rS_i\CH_0(M)$. For the last statement, we have  $$\rS_i\CH_0(M) \subseteq \rS_i^{\rm SYZ} \CH_0(M)$$ from the definition of $\rS_i\CH_0(M)$. The assertion then  follows from the inclusion $ \rS_i^{\rm SYZ} \CH_0(M) \subseteq \rS_i\CH_0(M) $.
\end{proof}

\subsection*{Proof of Theorem \ref{mainthm}} We first deal with the case $M=M_H(v)$ with $r\geq 0$. By tensoring the sheaf in $M$ with $mH$ for $m$ sufficiently large, we may always assume that 

\begin{equation}\label{eq:larged}
    D\cdot H>\max\{ 4r^2+1, 2r(v^2 H^2-D^2 H^2+(D\cdot H)^2)\},  ~s>2d,  \hbox{~if $r>0$}
\end{equation}
or \begin{equation}\label{eq:larges}
    s>\max\{2d, \frac{(D\cdot H)^2}{2H^2}+1\}, \hbox{~if $r=0$}.
\end{equation}
Let $\mu_H$ be the slope function with respect to $H$. Under these conditions,  Yoshioka has shown in \cite[Theorem 1.7, Corollary 2.14]{Yoshi09}  that   there is an isomorphism 
$$\Phi: M_H(v)\to M_{\widehat{H}}(\widehat{v})$$
induced by a Fourier-Mukai transform such that for any stable sheaf $E\in M$, $\Phi(E)$ is a $\mu_{\widehat{H}}$-stable sheaf of rank $s$ (See also \cite[Theorem 3.4]{MR3942159}). The conditions also ensures that the $\mu_{\widehat{H}}$-stable sheaves in $M_{\widehat{H}}(\widehat{v})$  have sufficiently large rank and have to be locally free (cf.~\cite[Remark 3.2]{Yo01} and \cite[Remark 6.1.9]{Huy}). 
In fact, there is an exact sequence
$$0\rightarrow E\rightarrow E^{\vee\vee}\rightarrow Q\rightarrow 0$$
with $E^{\vee\vee}$ being $\mu_{\widehat{H}}$-stable  and $Q$ being a $0$-dimensional sheaf whose support is of length $l$. Then $v(E^{\vee\vee})^2=v(E)^2-2rl\geq -2$. Therefore $l=0$ if $r>\frac{v(E)^2+2}{2}$.
By Theorem \ref{thm:lf}, we have  $$\rS_\bullet^{\rm SYZ}(M_{\widehat{H}}(\widehat{v}))=\rS_{\bullet}(M_{\widehat{H}}(\widehat{v}))$$
Finally, according to \cite[Proposition 0.4]{SYZ20}, $\Phi$ induces an isomorphism between the two filtartions $$\rS_\bullet^{\rm SYZ}\CH_0(M_{H}(v))\cong \rS_{\bullet}^{\rm SYZ}\CH_0(M_{\widehat{H}}(\widehat{v})).$$
It follows that  $\rS_\bullet^{\rm SYZ}\CH_0(M)=\rS_{\bullet}\CH_0(M)$.

Now we return to the case $M=M_\sigma(v)$ is a Bridgeland moduli space.  According to \cite[Proposition 3.5]{MR3942159}, every Bridgeland moduli space  is isomorphic to a moduli space of Gieseker stable sheaves via a derived equivalence.  Then again by \cite[Proposition 0.4]{SYZ20}, one can immediately get  the assertion.

\begin{corollary}\label{cor:dense}
For any $E\in M$ with dim $O_E= d-i$, the subvarieties of the orbit $O_E$ with dimension $d-i$ are dense in $M$.
\end{corollary}
\begin{proof}
In \cite{SYZ20}, it has been shown that there is an incident variety 
$$R\subset \{(E,\xi)~|~c_2(E)=[\mathrm{supp}(\xi)]+c\cdot o_X\in \CH_0(X) \}\subset M\times S^{[d]}$$
such that the natural projections in the diagram below
\begin{equation*}
   \begin{tikzcd}
    & R \arrow[ld, "p"'] \arrow[rd, "q"] &           \\
    M &                                    & {X^{[d]}}
   \end{tikzcd} 
\end{equation*}
are both generically finite and surjective.  Here,  $c$ is a constant depending on the Mukai vector of $E$.
 By Theorem  \ref{mainthm}, we have  $c_2(E)\in \rS_i(X)$. Using the surjectivity of $p$, we can find $\xi\in X^{[d]}$ such that $(E,\xi)\in R$. As $c_2(\xi)\in \rS_i(X)$,   by \cite[Lemma 3.5]{Voi15}, the subvarieties $W_\alpha \subseteq O_{\xi}$ with dimension $d-i$ are dense in $X^{[d]}$. Shrinking  $R$ if necessary, we may assume  $p$ and $q$ are finite. Then the subvarieties  $p\circ q^{-1}(W_\alpha)$ are lying  in the orbit  $O_E$ of dimension $d-i$ and they are clearly Zariski dense in $M$. 
\end{proof}

\begin{corollary}
The filtration $S_\bullet\CH_0(M)$ is invariant under birational transform between hyper-K\"ahler varieties, i.e. if $f:M=M_\sigma(v)\dashrightarrow M'$ is a birational map, then the induced isomorphism $$\CH_0(M')\cong \CH_0(M)$$ 
preserves the filration $S_\bullet$. 
\end{corollary}
\begin{proof}
The hyper-K\"ahler birational models of $M$ is controlled by its birational ample cone. 
According to \cite[Theorem 1.2]{BM14}, there is a continuous map
$$\ell: {\rm Stab}^\dag(X)\rightarrow {\NS( M_\sigma(v))},$$
from a connected component of space of stability conditions on $\rD^b(X)$ to the N\'eron-Severi group such that the image of $\ell$ is the birational ample cone and  $M'$ is isomorphic to a moduli space of stable objects $M_{\sigma'}(v)$ for a generic $\sigma'$ in some chamber $\cC'$, such that $\ell(\cC')=f^*{\rm Amp}(M_{\sigma'}(v))$ and $\ell(\cC)={\rm Amp}(M_{\sigma}(v))$ for the chamber $\cC$ containing $\sigma$. Moreover, up to an automorphism, the birational map $f:M_\sigma(v) \dashrightarrow M_{\sigma'(v)}$ is actually induced by a derived (anti) autoequivalence. 

Thus we can deduce the result by the $\rS_\bullet=\rS^{\rm SYZ}_\bullet$ and that $\rS^{\rm SYZ}_{\bullet}$ is independent of  modular interpretations.

\end{proof}

\section{Applications: Comparison with other filtrations}

With the notations as before, on the moduli space $M=M_H(v)$ of stable sheaves, there is a distinguished point $c_M=[E_0]\in M$ such that $c_2(E_0)=ko_X$ with $k=\frac{D^2}{2}+r-s$. The existence of  $c_M$ is ensured by  the surjectivity of the projection map from the incident variety $R$ to $X^{[d]}$. Let us recall the filtrations defined in \cite{BFMS19}. 

\begin{definition}Let $\Delta$ be the diagonal in the product $M\times M$ and we set 
\begin{equation}
\overline{\Delta}=\Delta-M\times c_M \in \CH_{2d}(M\times M)
\end{equation}
Then one can define
\begin{equation}   
\rS^{\rm BFMS}_i(\CH_0(M)):= \{\alpha ~| ~\alpha \cdot \overline{\Delta}_{0,1} \ldots \cdot \overline{\Delta}_{0,i+1} = 0 \in  \CH_\ast (M \times M^{i+1})\},
\end{equation}
where $\overline{\Delta}_{0,j}$ is the pull back of $\overline{\Delta}$ from the $0$-th and $j$-th factors for $j=0,1,\cdots i+1$. It is useful to single out the subgroup 
$$\rS^{\mathrm{small}}_i(\CH_0(M))\subset \rS^{\rm BFMS}_i(\CH_0(M))$$
 generated by point classes $[F]\in \rS_i^{\rm BFMS}(\CH_0(M))$ with  $F\in M$.
\end{definition}

In \cite{BFMS19} and \cite{Vi20}, it has been proved that 
\begin{equation}\label{same}
       \rS^{\mathrm{small}}_{\bullet}\CH_0(M)=\rS^{\rm BFMS}_\bullet \CH_0(M)=\rS_\bullet^{\rm SYZ} \CH_0(M)=\rR_\bullet \CH_0(M),
\end{equation}
(cf.~\cite[Proposition A.6, Theorem 8.4]{Vi20}). Furthermore, for 0-cycles, the termination of this filtration implies that for all $F\in M$
$$[F]\cdot  \overline{\Delta}_{0,1} \cdot  \overline{\Delta}_{0,2}\ldots \overline{\Delta}_{0,d+1}= 0 \in \CH_*(M \times M^{d+1}).$$ 
Furthermore, they conjectured that a generalization of this equation for higher dimensional constant cycle varieties is still valid. As a consequence of Theorem \ref{mainthm}, we confirm their conjecture. 
\begin{theorem}\cite[Conjecture 4]{BFMS19}
Let $Y\subseteq M$ be a constant cycle subvariety. Then
 $$[Y]\cdot  \overline{\Delta}_{0,1} \cdot  \overline{\Delta}_{0,2}\ldots \overline{\Delta}_{0,d+1}= 0 \in \CH_*(M \times M^{d+1}).$$
Here, $[Y]$ is pulled back  as usual from the first $M$ factor.
\end{theorem}

\begin{proof}
According to \cite[Lemma 4]{BFMS19}, Conjecture holds if and only if $$\rS_\bullet\CH_0(M)\subseteq \rS^{\rm BFMS}_\bullet\CH_0(M).$$
The assertion follows directly from \eqref{same}
(cf.\cite[Lemma 3]{BFMS19} and \cite[Theorem 8.4]{Vi20}).
\end{proof}

Combining the above results, one can conclude Theorem \ref{mainthm2}.

\begin{remark}
It was asked in \cite[Question 3.2]{SYZ20}  whether one has the following equivalence $$[E]\in \rS_i^{\rm SYZ}\CH_0(M) ~{\rm if~and~only~if}~ c_2(E)\in\rS_i(X).$$
An affirmative answer of this question can be implied by the following statement: 
\begin{equation}\label{eq:nonvan}
[V_E]\cdot  \overline{\Delta}_{0,1} \cdot  \overline{\Delta}_{0,2}\ldots \overline{\Delta}_{0,d}\neq 0\end{equation}
for a maximal constant cycle variety $V_E$ containing $E$.   This is due to a direct computation in \cite[Remark 6]{BFMS19}: if $E\in \rS_i^{\rm small}\CH_0(X)$, one must have dim $V_E\geq d-i$ assuming \eqref{eq:nonvan}.  Then up to taking a derived equivalence  preserving O'Grady's filtration, one can assume that $E$ is locally free and obtain $c_2(E)\in \rS_i(X)$ by applying Theorem \ref{thm:lf}.

In \cite[Remark 6]{BFMS19}, it also claims that \eqref{eq:nonvan} holds if $$\rS_\bullet^{\mathrm{small}}\CH_0(M)\subseteq \rS_\bullet\CH_0(M).$$which is always true by Theorem \ref{mainthm2}. However, it seems to us that there exists a gap in the proof:  it's unclear why there  exists a constant cycle variety $V_E$  containing $E$ with dim $V_E\geq n-i$ for $E\in \rS_i^{\rm small}\CH_0(M) \cap \rS_i\CH_0(M)\setminus \rS_{i-1}^{\rm small}\CH_0(M)$. Hence the question seems still open. 
\end{remark}

\bibliographystyle{alpha}
\bibliography{main}

\end{document}